\documentclass[oneside,10pt]{article}          % please do not change
\usepackage[b5paper]{geometry}	    % your paper can be easily printed on a4 or letter paper with enlargenment      

\usepackage{graphicx}
\usepackage{enumerate}
\usepackage{epstopdf}
\usepackage{amsfonts,amsmath,latexsym,amssymb} % these packages are required
\usepackage{amsthm}                % please use one of this two options for theorems
\usepackage{mathrsfs,upref}         % not so essential, but part of journal style
\usepackage{mia}	            % journal style, comment only if you find some bug.

\usepackage[final]
{showkeys}

\usepackage{mathscinet}
\usepackage{calrsfs}
\usepackage{mathtools}

\usepackage{hyperref}

\newtheorem{theorem}{Theorem}[section]           
\newtheorem{lemma}[theorem]{Lemma}               

\newtheorem{proposition}[theorem]{Proposition}

\theoremstyle{definition}

\newtheorem{remark}[theorem]{Remark}

\numberwithin{equation}{section}       

\newcommand{\opt}{{\operatorname{opt}}}

\newcommand{\si}{\sigma}

\newcommand{\ka}{\kappa}
\newcommand{\la}{\lambda}

\newcommand{\de}{\delta}

\renewcommand{\Psi}{\overline{\Phi}}

\newcommand{\BH}{{\operatorname{\mathsf{BH}}}}

\newcommand{\ii}[1]{\,\operatorname{I}\{#1\}} 
\renewcommand{\P}{\operatorname{\mathsf{P}}} 
\newcommand{\E}{\operatorname{\mathsf{E}}}
\newcommand{\Var}{\operatorname{\mathsf{Var}}}

\newcommand{\R}{\mathbb{R}}

\renewcommand{\L}{\mathcal{L}}

\newcommand{\vp}{\varepsilon}
\newcommand{\tz}{{\tilde{z}}}
\begin{document}

\title[Exact converses to a reverse AM--GM inequality]{Exact converses to a reverse AM--GM inequality, with applications to sums of independent random variables and (super)martingales}

% Short title is optional, it will appear in running heads.
% It is necessary only if the title is to long to be used in running heads

\author{Iosif Pinelis}

\address{Department of Mathematical Sciences\\
Michigan Technological University\\
Houghton, Michigan 49931, USA\\
\email{ipinelis@mtu.edu}}

%\address{Second Author, Full postal Address of the Second Author,\\
%\email{email of the Second Author}}
%
%\address{Third Author, Full postal Address of the Third Author\\
%\email{email of the Third Author}}

\CorrespondingAuthor{Iosif Pinelis}

%\dedicated{Dedicated to...}                    % Optional

\date{\today}                               % Please, write the date of submission


        % AMS-2010 subj class. The list can be found on http://www.ams.org/mathscinet/msc/msc2010.html

%26-XX			Real functions [See also 54C30] 
%26D15  	Inequalities for sums, series and integrals       
%        26D07  	Inequalities involving other types of functions
%60E15  	Inequalities; stochastic orderings        

%\thanks{Supported in part by NSF grant DMS-0805946 and NSA grant H98230-12-1-0237} 
        % Optional. Only one command thanks is allowed, use \par inside text if you need multiple thanks.       

\begin{abstract}
For every given real value of the ratio $\mu:=A_X/G_X>1$ of the arithmetic and geometric means of a positive random variable $X$ and every real $v>0$, exact upper bounds on the right- and left-tail probabilities $\P(X/G_X\ge v)$ and $\P(X/G_X\le v)$ are obtained, in terms of $\mu$ and $v$. In particular, these bounds imply that $X/G_X\to1$ in probability as $A_X/G_X\downarrow1$. 
Such a result may be viewed as a converse to a reverse Jensen inequality for the strictly concave function $f=\ln$, whereas the well-known Cantelli and Chebyshev inequalities may be viewed as converses to a reverse Jensen inequality for the strictly concave quadratic function $f(x)\equiv -x^2$. As applications of the mentioned new results, improvements of the Markov, Bernstein--Chernoff, sub-Gaussian, and Bennett--Hoeffding probability inequalities are given. 
%to sums of independent random variables and (super)martingales
%
%\\
%!!!!add more to this 
\end{abstract}

\maketitle

%%%%% END OF TITLE PAGE %%%%%%%%%%%%%%%%%%%%%%%%%%%%%%%%%%%%%%%%%%%%%%

%%%%% BODY OF THE PAPER %%%%%%%%%%%%%%%%%%%%%%%%%%%%%%%%%%%%%%%%%%%%%%
% You should eventually delete (after reading) the rest of the text below %% 

\section{Introduction}\label{intro}
%https://mathoverflow.net/questions/377913/probability-of-a-deviation-when-jensen-s-inequality-is-almost-tight

Let $X$ be a positive random variable (r.v.). One can define the arithmetic and geometric means of $X$ as follows: 
\begin{equation}\label{eq:AX,GX}
	A_X:=\E X\quad\text{and}\quad G_X:=\exp\E\ln X,
\end{equation}
assuming that $\E X$ and $\E\ln X$ exist and are finite. 
% (which will always be the case as long as $\E X<\infty$). 

Consider the special case when, for given positive real numbers $x_1,\dots,x_n$, the distribution of the r.v.\ $X$ is defined by the formula 
\begin{equation}\label{eq:discr}
	\E f(X)=\frac1n\,\sum_{i=1}^n f(x_i)\quad\text{for any function $f\colon\R\to\R$.}
\end{equation}
\big(So, in the case when the numbers $x_1,\dots,x_n$ are pairwise distinct, any such r.v.\ $X$ takes each of the values $x_1,\dots,x_n$ with probability $\frac1n$.\big) 
In this case, 
\begin{equation}\label{eq:a-g}
A_X=\E X=%\ol x:=
\frac{x_1+\dots+x_n}n\quad\text{and}\quad 
G_X=\exp\E\ln X= %x^\g:=
\sqrt[%\leftroot{3}\uproot{3}
^n]{x_1\cdots x_n}. 
\end{equation}
Thus, the definitions \eqref{eq:AX,GX} of the arithmetic and
geometric means of a r.v.\ $X$ generalize the usual definitions of the arithmetic and geometric means of finitely many positive real numbers.

Since any bounded positive r.v.\ can be approximated in distribution by uniformly bounded r.v.'s each taking finitely many positive real values with equal probabilities, 
%such as the r.v.\ $X$ described in the beginning of this paragraph, 
the exact bounds to be stated in Theorem~\ref{th:} will remain exact in an appropriate sense if one considers only the r.v.'s with such discrete uniform distributions. 
%In particular, one has the following immediate corollary from Theorem~\ref{th:} and Remark~\ref{rem:le}.

The arithmetic mean--geometric mean (AM--GM) inequality 
\begin{equation}\label{eq:am-gm-ineq}
	A_X\ge G_X
\end{equation}
is a special case (with $f=\ln$) of Jensen's inequality 
\begin{equation}\label{eq:jensen}
	f(\E X)\ge\E f(X) 
\end{equation}
for concave functions $f$. 

Clearly, if the r.v.\ $X$ is constant almost surely (a.s.) -- that is, if $\P(X=c)=1$ for some real $c>0$, then the Jensen inequality \eqref{eq:jensen} and, in particular, the AM--GM inequality \eqref{eq:am-gm-ineq} turn into the equalities. Therefore, one may expect that, if the r.v.\ $X$ is close to a constant in some sense, then both sides of the Jensen inequality will be close to each other and, in particular, the arithmetic and geometric means of the r.v.\ $X$ will be close to each other. 

There are indeed a large number of theorems in this vein, called \emph{reverse Jensen inequalities}; see e.g.~\cite{%jebara-pentland
budimir-etal}. Usually, in such theorems the condition of $X$ being close to a constant is that the values of $X$ are in a bounded interval $[m_X,M_X]$, which latter may be thought of as small, with the conclusion that the difference $f(\E X)-\E f(X)$ between the left- and right-hand sides of the Jensen inequality \eqref{eq:jensen} is 
%bounded by a bound, which is 
small if the interval $[m_X,M_X]$ is small. %; see e.g.\ \cite{simic} and references therein. 
Somewhat related results were obtained in \cite{arithm-geom_publ}.  

%In distinction from such reverse Jensen inequalities for arbitrary concave functions $f$, in the special case of $f=\ln$ an exact upper bound on the difference $A_X-G_X$ between the arithmetic and geometric means of $X$ was obtained in \cite{arithm-geom_publ}, which may be small even when $M_X$ is large (or infinite). \big(An exact lower bound on $A_X-G_X$ was also given in \cite{arithm-geom_publ}, which thus provides a refinement of the AM--GM inequality.\big)

%the inequality \eqref{eq:jensen} and, in particular, AM--GM inequality \eqref{eq:am-gm-ineq}

Note further that, if the function $f$ is strictly concave, then the equality in \eqref{eq:jensen} implies that the r.v.\ $X$ is a.s.\ constant. Therefore, it appears natural to inquire whether statements of the following form hold: If the two sides of the Jensen inequality \eqref{eq:jensen} with a strictly concave function $f$ are close to each other, then the r.v.\ $X$ is close to a constant in some sense. Such a statement may be referred to as a \emph{converse to a reverse Jensen inequality}. 
% or simply as a \emph{converse reverse Jensen inequality}. 

Converses to reverse Jensen inequalities are very well known and very widely used in the case when $f(x)\equiv -x^2$. Then the difference between the left- and right-hand sides of \eqref{eq:jensen} is $\si^2:=\Var X$, the variance of $X$. In this case, one has Cantelli's inequality 
\begin{equation}\label{eq:cant}
	\P(X-\mu\ge\vp)\vee\P(X-\mu\le-\vp)\le\frac{\si^2}{\si^2+\vp^2}
\end{equation}
and 
Chebyshev's inequality 
\begin{equation}\label{eq:cheb}
	\P(|X-\mu|\ge\vp)\le\frac{\si^2}{\vp^2}
\end{equation}
for all real $\vp>0$, where $\mu:=\E X\in\R$ and $a\vee b:=\max(a,b)$. The Cantelli and Chebyshev bounds are exact in their terms. In particular, \eqref{eq:cant} turns into the equality when $\P(X=\mu+\vp)=\vp^2/(\si^2+\vp^2)=1-\P(X=\mu-\si^2/\vp)$ or when $\P(X=\mu-\vp)=\vp^2/(\si^2+\vp^2)=1-\P(X=\mu+\si^2/\vp)$, whereas \eqref{eq:cheb} turns into the equality when $\P(X=\mu+\vp)=\P(X=\mu-\vp)=1/2$. 

So, for any given real $\vp>0$, if $f(x)\equiv -x^2$ and the difference $\si^2$ between the left- and right-hand sides of \eqref{eq:jensen} is small enough, then $X$ deviates from the constant $\mu$ with a however small probability. Thus, the Cantelli and Chebyshev inequalities are indeed converses to a reverse Jensen inequality for $f(x)\equiv -x^2$. 

In this paper, we shall provide converses to reverse Jensen inequalities for $f=\ln$, that is, converses to reverse AM--GM inequalities. This case appears to be the next in importance after the Chebyshev--Cantelli ``quadratic'' case of $f(x)\equiv-x^2$ -- see the applications to %Markov's inequality and to 
the so-called exponential bounds on the tails of the distributions of sums of independent r.v.'s in Section~\ref{exp-bounds}; here one may also note  
e.g.\ \cite[Lemma~3.9]{buckley}. Just as the Cantelli and Chebyshev bounds, our bounds are exact in their own terms. However, the case of $f=\ln$ is much more difficult than that of $f(x)\equiv-x^2$. 

\section{Basic results and discussion
}\label{results}

The main result of this paper is as follows. 

\begin{theorem}\label{th:}
Let $X$ be a positive r.v.\ with finite $\E X$ and $\E\ln X$. Suppose that $\P(X=c)<1$ for each real $c$, so that  
\begin{equation}\label{eq:EX,ElnX}
	\mu:=%\E X\in(1,\infty)\quad\text{and}\quad\E\ln X=0. 
	\frac{A_X}{G_X}>1. 
\end{equation}
Then 
\begin{enumerate}[(I)]
	\item \label{I}
\begin{equation}\label{eq:right}
	\P\Big(\frac X{G_X}\ge v\Big)\le p_{\mu,v}:=p_v:=\frac{\mu-z_v}{v-z_v}\in(0,1) \quad\text{for each}\quad v\in(\mu,\infty)
\end{equation}
and 
\begin{equation}\label{eq:left}
	\P\Big(\frac X{G_X}\le v\Big)\le p_v\in(0,1) \quad\text{for each}\quad v\in(0,1),
\end{equation}
where, for each $v\in(\mu,\infty)$, $z_v=z_{\mu,v}$ is the only root $z\in(0,1)$ of the equation 
\begin{equation}\label{eq:F}
	F(z):=F_{\mu,v}(z):=(v-\mu)\ln z+(\mu-z)\ln v=0
\end{equation}
and, for each $v\in(0,1)$, $z_v=z_{\mu,v}$ is the only root $z\in(\mu,\infty)$ of equation \eqref{eq:F}. 

%!!! $z_v$ in terms of Lambert's $W$ function
\item \label{II} For each $v\in(\mu,\infty)$ and for each $v\in(0,1)$, the upper bound $p_v$ in the corresponding inequalities in \eqref{eq:right} and \eqref{eq:left} is exact, as it is attained when 
\begin{equation}\label{eq:X_v}
	\P(X=v)=p_v=1-\P(X=z_v),
\end{equation}
and for such a r.v.\ $X$ the condition $A_X/G_X=\mu$ holds -- cf.\ \eqref{eq:EX,ElnX}. 
%obviously holds. 
\item \label{III} We have 
\begin{equation}\label{eq:rho=1}
	\rho_\mu(v):=\sup_{A_X/G_X=\mu}\P\Big(\frac X{G_X}\ge v\Big)=1
	\quad\text{for each}\quad v\in(-\infty,\mu]
\end{equation}
and 
\begin{equation}\label{eq:la=1}
	\la_\mu(v):=\sup_{A_X/G_X=\mu}\P\Big(\frac X{G_X}\le v\Big)=1
	\quad\text{for each}\quad v\in[1,\infty),
\end{equation}
where $\sup%\limits
_{A_X/G_X=\mu}$ denotes the supremum over all positive r.v.'s $X$ with finite $\E X$ and $\E\ln X$ and with $A_X/G_X=\mu$. 
In particular, for each $v\in[1,\mu]$, the exact upper bound on either one of the two tail probabilities, $\P\big(\frac X{G_X}\ge v\big)$ and $\P\big(\frac X{G_X}\le v\big)$, is $1$; it is not attained, though.  
%??? what about bound on $\P\big(\frac X{G_X}\ge v\big)$ for $v\in[1,\mu)$ ?? \\ 
\item \label{IV} One also has the following simple (but not exact) upper bounds on $\P\big(\frac X{G_X}\ge v\big)$ and $\P\big(\frac X{G_X}\le v\big)$: 
\begin{equation}\label{eq:right,q}
	\P\Big(\frac X{G_X}\ge v\Big)\le q_{\mu,v}:=q_v:=\min\Big(1,\frac{\mu-1}{v-1-\ln v}\Big) \quad\text{for each}\quad v\in[1,\infty)
\end{equation}
(with $q_1:=1$) and 
\begin{equation}\label{eq:left,q}
	\P\Big(\frac X{G_X}\le v\Big)\le q_v 
	%\frac{\mu\ln\mu}{v-\mu+\mu\ln(\mu/v)}
	%\underset{\mu\downarrow1}\widesim
	%\frac{\mu-1}{v-1-\ln v} 
	\quad\text{for each}\quad v\in(0,\mu]. 
\end{equation}
%\reverse-reverse-jensen\Mathematica\Untitled-17.nb
\item \label{V} 
%\begin{remark}\label{rem:le mu}
The condition $\mu:=A_X/G_X$ in %Theorem~\ref{th:}
\eqref{eq:EX,ElnX} can be replaced by the $A_X/G_X\le\mu$. 
%This follows from the proof of Theorem~\ref{th:} (in Section~\ref{proofs}). More specifically, given only the condition $A_X/G_X\le\mu$ (which means that $\E X\le\mu$ when \eqref{eq:G_X=1} is assumed), the second equality sign in \eqref{eq:P(X>v)<R} can be replaced by $\le$, since $a>0$. So, the inequality $\P(X\ge v)\le R_z(v)$ will continue to hold when $0<z<v$. Similarly, \eqref{eq:P(X<v)<R} will continue to hold. %, under the assumption $A_X/G_X\le\mu$. 
%\end{remark}
\end{enumerate}
\end{theorem}

\begin{remark}\label{rem:2-set}
Part~\eqref{II} of Theorem~\ref{th:} shows that, 
as in the cases of the Cantelli and Chebyshev inequalities, the ``extreme'' r.v.'s $X$ providing the attainment in our inequalities \eqref{eq:right} and \eqref{eq:left} take only two values.  
\end{remark}

\begin{remark}\label{rem:concentr}
Inequalities \eqref{eq:right,q} and \eqref{eq:left,q} imply concentration of the r.v.\ $X$ near its (say) geometric mean $G_X$ when the arithmetic mean $A_X$ is close to $G_X$. More precisely, we have  $X/G_X\to1$ in probability as $\mu=A_X/G_X\downarrow1$. 
Thus, Theorem~\ref{th:} indeed provides converses to the reverse Jensen inequality for $f=\ln$. 
\end{remark}
 
Remark~\ref{rem:concentr} is illustrated in Figure~\ref{fig:pic}. 
\begin{figure}[h]
	\centering
		\includegraphics[width=1.00\textwidth]{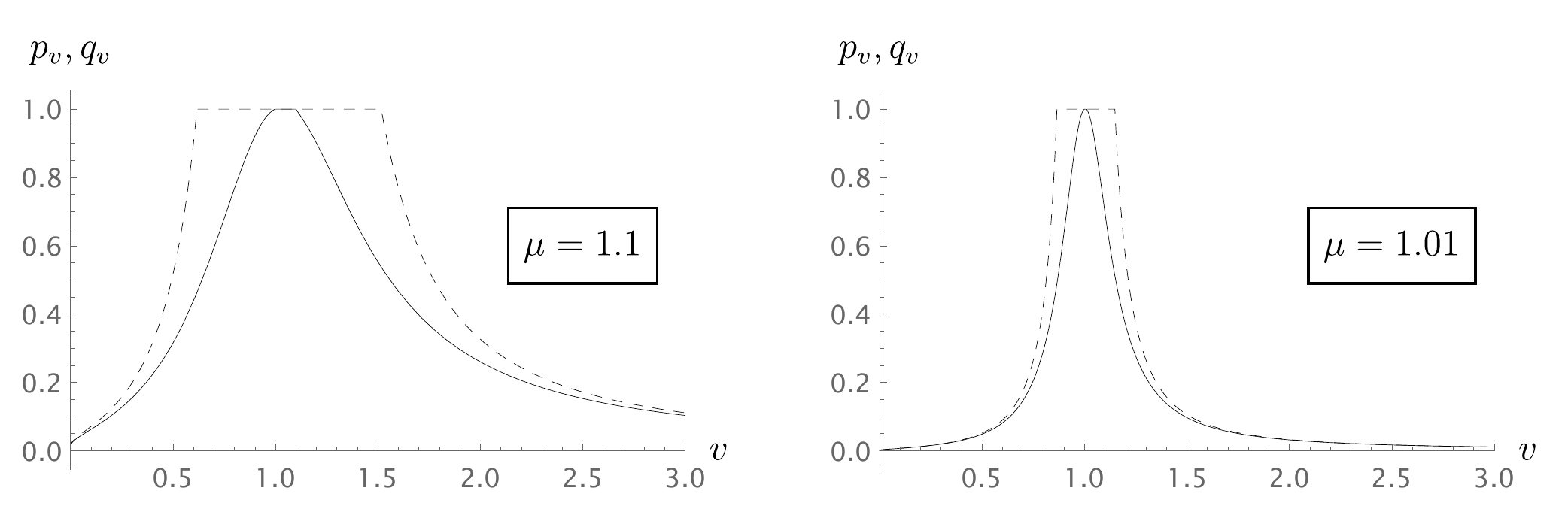}
	\caption{Graphs $\{(v,p_v)\colon0<v<3\}$ (solid) and $\{(v,q_v)\colon0<v<3\}$ (dashed) for $\mu=1.1$ (left) and $\mu=1.01$ (right), with $p_v:=1$ for $v\in[1,\mu]$.}
	\label{fig:pic}
\end{figure}

At least in the case when the distribution of the r.v.\ $X/G_X$ is highly concentrated (that is, when $\mu$ is close to $1$),  
the simple bound $q_v$ on the tails of the distribution of the r.v.\ $X/G_X$ is not too far from the exact bound $p_v$ when $v$ is somewhat close to $1$ but $q_v<1$ (so that $v$ is not too close to $1$). More precisely, we have the following proposition: 

\begin{proposition}\label{prop:p,q}
Suppose that $\mu$ and $v$ both go to $1$ in any way such that $v\in(\mu,\infty)\cup(0,1)$ and 
$q_v$ is less than $1$. 
% but bounded away from $0$. 
Then 
\begin{equation}\label{eq:p,q}
	p_v-\frac{q_v}{1+q_v}\to0. 
\end{equation}
\end{proposition}
 
\begin{proposition}\label{prop:L}
For $v\in(\mu,\infty)\cup(0,1)$, one has the following expression for the root $z_v$ of equation~\eqref{eq:F}: 
\begin{equation}\label{eq:z=}
	z_v=z_{\mu,v}=-\tz_v W_{\ka(v)}\Big(-\frac{e^{-\mu/\tz_v}}{\tz_v}\Big),
\end{equation}
where 
\begin{equation}\label{eq:tz_v}
	\tz_v:=\frac{v-\mu}{\ln v},
\end{equation}
\begin{equation*}
	\ka(v):=\begin{cases}
	0 &\text{ if }v\in(\mu,\infty),\\
	-1 &\text{ if }v\in(0,1),
	\end{cases}
\end{equation*}
and $W_k$ is the $k$th branch of Lambert's $W$ function \cite{knuth96}, so that 
%reverse-reverse-jensen\Mathematica\W.nb
\begin{enumerate}[(i)]
	\item for all $t\in(-1,\infty)$ and $u\in(-1/e,\infty)$, we have $te^t=u\iff t=W_0(u)$; 
	\item for all $t\in(-\infty,-1)$ and $u\in(-1/e,0)$, we have $te^t=u\iff t=W_{-1}(u)$.  
\end{enumerate}
\end{proposition}

%\begin{remark}\label{rem:le mu}
%The condition $A_X/G_X=\mu$ in Theorem~\ref{th:} can be replaced by the $A_X/G_X\le\mu$. This follows from the proof of Theorem~\ref{th:} (in Section~\ref{proofs}). More specifically, given only the condition $A_X/G_X\le\mu$ (which means that $\E X\le\mu$ when \eqref{eq:G_X=1} is assumed), the second equality sign in \eqref{eq:P(X>v)<R} can be replaced by $\le$, since $a>0$. So, the inequality $\P(X\ge v)\le R_z(v)$ will continue to hold when $0<z<v$. Similarly, \eqref{eq:P(X<v)<R} will continue to hold. %, under the assumption $A_X/G_X\le\mu$. 
%\end{remark}

\section{Applications: Improvements of Markov's bound and exponential bounds on the tails of the distributions of sums of independent r.v.'s and (super)martingales}\label{exp-bounds}

\subsection{Improvements of the Markov bound and of the Bernstein--Chernoff bound}\label{markov,Be-Ch}
%\begin{remark}\label{rem:markov}
By Markov's inequality, with $\mu$ as in Theorem~\ref{th:}, 
\begin{equation}\label{eq:markov}
	\P\Big(\frac X{G_X}\ge v\Big)\le\frac\mu v
\end{equation}
for all real $v>0$ (this inequality is nontrivial only if $v>\mu$). 

The bound $p_{\mu,v}=\dfrac{\mu-z_v}{v-z_v}$ in \eqref{eq:right} is a (best possible) improvement of the Markov bound $\dfrac\mu v$ in
\eqref{eq:markov} -- because $z_v(<1)<\mu<v$. 
%
%Inequality \eqref{eq:right} in Theorem~\ref{th:} is a (best possible) improvement of Markov's inequality \eqref{eq:markov} -- because $z_v(<1)<\mu<v$. 
Even though Markov's inequality is well-known (and easy to see) to be exact in its terms, the just mentioned improvement has been possible by taking into account that the geometric mean of the r.v.\ $X/G_X$ is $1$. 
This improvement over Markov's inequality may be dramatic in some cases. Indeed, when e.g.\ $\mu(>1)$ is close to $1$ while $v(>1)$ is not close to $1$, then even the suboptimal bound $%q_v
\dfrac{\mu-1}{v-1-\ln v}$ in \eqref{eq:right,q} will be much less than the Markov bound $\dfrac\mu v$. % in \eqref{eq:markov}. 
Similarly, inequality \eqref{eq:left} is a best possible, and in some settings dramatic, improvement of the corresponding left-tail Markov inequality.  

Take now any r.v.\ $Y$ with 
\begin{equation}\label{eq:EY=0}
	\E Y=0,
\end{equation}
any real number $y$, and any positive real number $\la$. 
The so-called Bernstein--Chernoff inequality
\begin{equation}\label{eq:Be-Che}
	\P(Y\ge y)\le \frac{\E e^{\la Y}}{e^{\la y}}
%	e^{-\la y}\E e^{\la Y},
\end{equation}
is a particular case of Markov's inequality \eqref{eq:markov}, with  
\begin{equation}\label{eq:X,v}
	X:=e^{\la Y}\quad\text{and}\quad v:=e^{\la y}.
\end{equation}
Also, the condition \eqref{eq:EY=0} implies that here
\begin{equation*}
	G_X=1. 
\end{equation*}
Actually, the Bernstein--Chernoff inequality \eqref{eq:Be-Che} is, not only a special case of Markov's inequality \eqref{eq:markov}, but of course also a restatement of \eqref{eq:markov}. In particular, just as Markov's inequality \eqref{eq:markov} does not take into account the fact that the geometric mean of $X/G_X$ is $1$, the Bernstein--Chernoff inequality
\eqref{eq:Be-Che} does not take condition \eqref{eq:EY=0} into account. 

Therefore, one can use Theorem~\ref{th:} to improve, not only Markov's inequality \eqref{eq:markov}, but also its equivalent, the Bernstein--Chernoff inequality \eqref{eq:Be-Che}. 

%Just as inequality \eqref{eq:right} in Theorem~\ref{th:} is a (best possible) improvement of Markov's inequality \eqref{eq:markov} 
%
%Markov's inequality immediately implies the so-called Bernstein--Chernoff inequality
%\begin{equation}\label{eq:Be-Che}
%	\P(Y\ge y)\le e^{-\la y}\E e^{\la Y},
%\end{equation}
%where $Y$ is any r.v., $y$ is any real number, and $\la$ is any positive real number. 

%Consider this thesis in more detail. 
When the r.v.\ $Y$ has an additional structure, 
one can obtain an upper bound $B(\la)$ on $\E e^{\la Y}$, and then $\inf_{\la\ge0}e^{-\la y}B(\la)$ will be an upper bound -- referred to as an exponential bound -- on the tail probability $\P(Y\ge y)$. A general approach to obtaining best possible exponential bounds of this kind, along with a number of specific results, in the case when $Y$ is the sum of independent r.v.'s was presented in \cite{pin-utev-exp}. 
Details on what has been said in this paragraph are provided in the following two subsections. 

\subsection{Improvements of the exponential bound in the sub-Gaussian case}\label{subG}
Suppose %indeed 
that 
\begin{equation}\label{eq:Y}
	Y=Y_1+\dots+Y_n,
\end{equation}
where $Y_1,\dots,Y_n$ are independent zero-mean r.v.'s. 
%the sum of independent zero-mean r.v.'s $Y_1,\dots,Y_n$. 

%Consider first 
In this subsection, we will consider the particularly simple case when the $Y_i$'s are sub-Gaussian, that is, when 
\begin{equation}\label{eq:subG}
	\E e^{\la Y_i}\le e^{\la^2\si_i^2/2}
\end{equation}
for some positive real numbers $\si_1,\dots,\si_n$, all $i\in[n]:=\{1,\dots,n\}$, and real $\la\ge0$. 
If $Y_i\sim N(0,\si_i^2)$ for all $i\in[n]$, then the sub-Gaussianity condition \eqref{eq:subG} holds with the equality sign. 
Also, for instance, \eqref{eq:subG} holds when $|Y_i|\le\si_i$ for all $i\in[n]$; cf.\ e.g.\ \cite[inequality~(4.16)]{hoeff63}. 

The constants $\si_i^2$ in \eqref{eq:subG} are referred to as (obviously, never unique) sub-Gaussian proxy variances of the corresponding r.v.'s $Y_i$. Clearly then, 
\begin{equation}\label{si^2}
	\si^2:=\si_1^2+\dots+\si_n^2
\end{equation}
is a sub-Gaussian proxy variance of the sum $Y$: 
\begin{equation}\label{eq:subG,Y}
	\E e^{\la Y}\le e^{\la^2\si^2/2}
\end{equation}
for all real $\la\ge0$. 

Take any real $y\ge0$. 
Then, by \eqref{eq:Be-Che}, 
\begin{equation}\label{eq:subG-bound}
	\P(Y\ge y)\le\inf_{\la\ge0}\frac{e^{\la^2\si^2/2}}{e^{\la y}}
	=\frac{e^{\la_y^2\si^2/2}}{e^{\la_y y}}
=P_1(t):=
	e^{-t^2/2}, 
\end{equation}
where
\begin{equation*}
	\la_y:=y/\si^2,\quad t:=y/\si,
\end{equation*}
and $\si:=\sqrt{\si^2}$. 

Using Theorem~\ref{th:}, %(and Remark~\ref{rem:le mu}), 
one can immediately improve the upper bound $e^{-t^2/2}$ on $\P(Y\ge y)$ in \eqref{eq:subG-bound}: % for all real $y>0$: 

\begin{proposition}\label{prop:subG}
For all real $y>0$, 
\begin{equation}\label{eq:P_2}
	\P(Y\ge y)\le %B_\opt(t):=
	P_2(t):=p_{\mu_t,v_t}=\frac{\mu_t-z_{\mu_t,v_t}}{v_t-z_{\mu_t,v_t}}
	<\frac{\mu_t}{v_t}=e^{-t^2/2}=P_1(t), 
\end{equation}
where 
\begin{equation}\label{eq:mu_t,v_t}
\mu_t:=e^{\la_y^2\si^2/2}=e^{y^2/(2\si^2)}=e^{t^2/2},\quad%\text{and}\quad
v_t:=e^{\la_y y}=e^{y^2/\si^2}=e^{t^2},   	
\end{equation}
and $z_{\mu,v}$ is as defined in part~\eqref{I} of 
Theorem~\ref{th:} or, equivalently, as in formula \eqref{eq:z=}. 
\end{proposition}

Concerning the inequality in \eqref{eq:P_2}, recall the reasoning in the second paragraph of Subsection~\ref{markov,Be-Ch}.

So, the bound $P_2(t)$ on $\P(Y\ge y)$ in \eqref{eq:P_2} improves the bound $P_1(t)$ in \eqref{eq:subG-bound} for all real $t>0$ or, equivalently, for all real $y>0$. 
To get the bound $P_2(t)$, we borrowed the minimizer $\la_y$ of the bound $\dfrac{e^{\la^2\si^2/2}}{e^{\la y}}$ on $\P(Y\ge y)$ and used $\la_y$ in the definitions of $\mu_t$ and $v_t$ in \eqref{eq:mu_t,v_t}. While this choice of $\la$ is optimal for the Markov bound $\dfrac{e^{\la^2\si^2/2}}{e^{\la y}}$, it will not be optimal for the better bound of the form $p_{\mu,v}$ based on Theorem~\ref{th:}. 

So, we can improve the bound $P_2(t)$ on $\P(Y\ge y)$ -- and thus further improve the bound $P_1(t)$ -- by avoiding the mentioned borrowing, as follows: 

\begin{proposition}\label{prop:subG1}
For all real $y>0$, 
\begin{equation}\label{eq:P_opt}
	\P(Y\ge y)\le %B_\opt(t):=
	P_\opt(y):=P_\opt(\si,y):=\inf_{\la>0}p_{\mu_\la,v_y(\la)}, 
\end{equation}
where 
\begin{equation}\label{eq:mu_la,v_la}
\mu_\la:=e^{\la^2\si^2/2}\quad\text{and}\quad
v_y(\la):=e^{\la y}.   	
\end{equation}
\end{proposition}

The drawback of the optimal bound $P_\opt(y)$ is that its expression in \eqref{eq:P_opt} is implicit; also, in distinction with the simpler bounds $P_1(t)$ and $P_2(t)$, $P_\opt(y)=P_\opt(\si,y)$ will depend on $\si,y$ not only through the simple ratio $t=y/\si$. 

On the other hand, clearly we can use the simple bound $q_{\mu,v}$ in \eqref{eq:right,q} to immediately get the following: 

\begin{proposition}\label{prop:subG2}
For all real $y>0$, 
\begin{equation}\label{eq:P_3}
	\P(Y\ge y)\le %B_\opt(t):=
	P_3(t):=q_{\mu_t,v_t}=\min\Big(1,\frac{\mu_t-1}{v_t-1-\ln v_t}\Big) 
	=\min\Big(1,\frac{e^{t^2/2}-1}{e^{t^2}-1-t^2}\Big), 
\end{equation}
where $\mu_t$ and $v_t$ are as in \eqref{eq:mu_t,v_t}. 
\end{proposition}

We see that the bound $P_3(t)$ is quite explicit and almost as simple as the bound $P_1(t)=e^{-t^2/2}$ in \eqref{eq:subG-bound}. Moreover, a simple algebra shows that $P_3(t)<P_1(t)$ (for a real $t>0$) if and only if $1+t^2<e^{t^2/2}$, that is, if and only if 
$t>t_*:=\sqrt{-2 W_{-1}\big(-1/(2 \sqrt e\,)\big)-1}=1.585\dots$, where, as in Proposition~\ref{prop:L}, $W_k$ denotes the $k$th branch of Lambert's $W$ function. Also, $P_1(t_*)=P_3(t_*)=0.284\dots$, which is substantially greater than commonly used values of the level of significance in statistical testing. So, the bound $P_3(t)$ is an improvement of the bound $P_1(t)$ for values of $t$ relevant in statistics. 

(Parts of) the graphs of the ratios of the bounds $P_2(t)$ in \eqref{eq:P_2}, $P_3(t)$ in \eqref{eq:P_3}, and $P_\opt(\si,\si t)$ in \eqref{eq:P_opt} with $\si=6$ to the baseline sub-Gaussian bound $P_1(t)$ in \eqref{eq:subG-bound} are shown 
in Figure~\ref{fig:subG}. 
\begin{figure}[h]
	\centering
		\includegraphics[width=.55\textwidth]{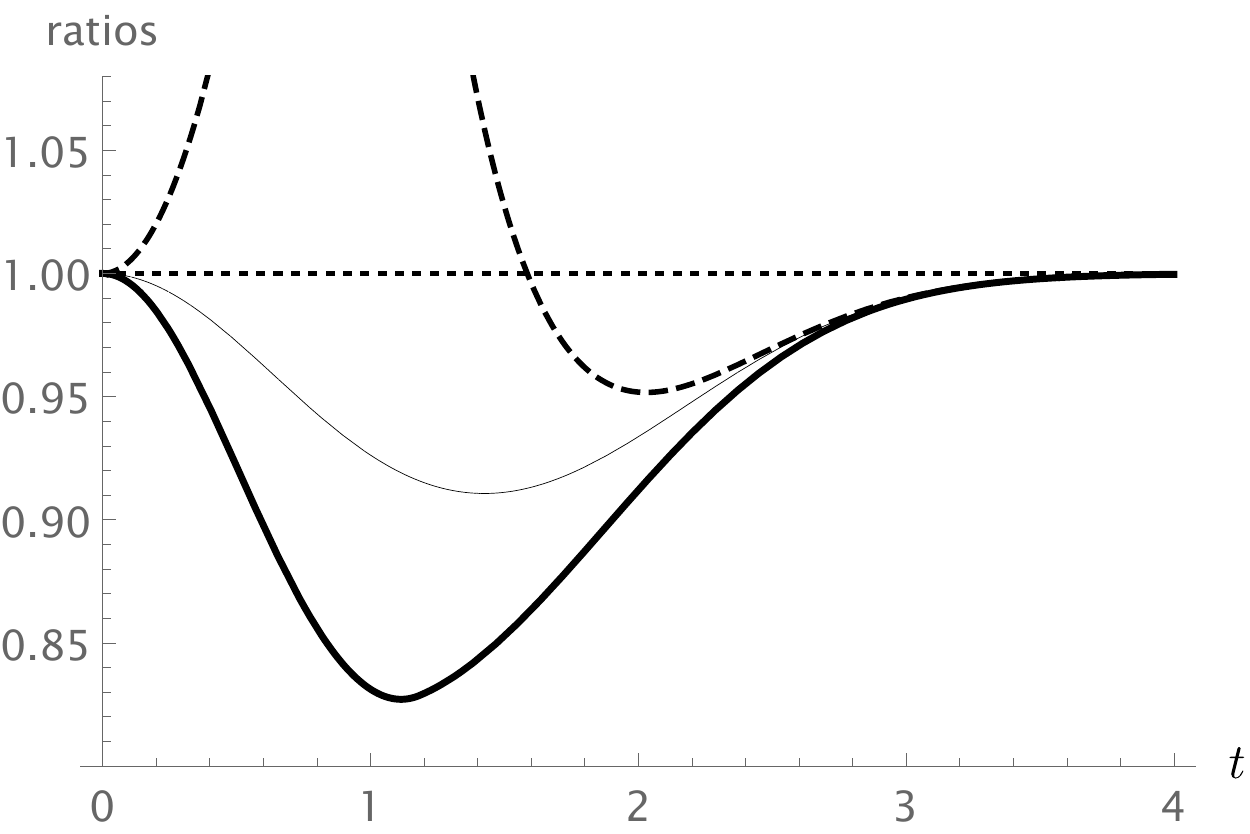}
	\caption{Graphs $\{(t,P_2(t)/P_1(t))\colon0<t<4\}$ (thin), 
	$\{(t,P_3(t)/P_1(t))\colon\break 
	0<t<4,P_3(t)/P_1(t)<1.08\}$(dashed), and 
	$\{(t,P_\opt(6,6t)/P_1(t))\colon0<t<4\}$ (thick).}
	\label{fig:subG}
\end{figure}

%%%-----
%%%
%%%
%%%%
%%%With $X$ and $v$ as in \eqref{eq:X,v}, and $\la=\la_y$, one can rewrite the inequality in \eqref{eq:subG-bound} in the following ``Markov'' form: 
%%%\begin{equation}\label{eq:markov,1}
%%%	\P(Y\ge y)=\P(X\ge v_t)\le B_\Markov(t):=\frac{\mu_t}{v_t},
%%%\end{equation}
%%%with 
%%%\begin{equation}
%%%v_t:=e^{\la_y y}=e^{y^2/\si^2}=e^{t^2}\quad\text{and}\quad
%%%\mu_t:=e^{\la_y^2\si^2/2}=e^{y^2/(2\si^2)}=e^{t^2/2}; 	
%%%\end{equation}
%%%then, in view of \eqref{eq:EY=0} and \eqref{eq:subG,Y}, $\mu$ is an upper bound on $\E(X/G_X)=\E X=\E e^{\la_y Y}$. So, by part~(I) of 
%%%Theorem~\ref{th:}, 
%%%\begin{equation}
%%%	\P(Y\ge y)=\P(X\ge v_t)\le B_\opt(t):=\frac{\mu_t-z_{\mu_t,v_t}}{v_t-z_{\mu_t,v_t}}
%%%	<\frac{\mu_t}{v_t}=B_\Markov(t). 
%%%\end{equation}

%Next, let  
%\begin{equation}
%	X:=e^{\la Y}\quad\text{and}\quad v:=e^{\la y}.
%\end{equation}
%Then the condition that 
%
%
%In this derivation of the bound $e^{-t^2/2}$ on $\P(Y\ge y)$, the condition that the $Y_i$'s are zero-mean was not taken into account. 

%\end{remark}

%\begin{remark}\label{rem:exp ineqs}
%!!! exp ineqs
%\end{remark}

%!!! approximations to the Lambert function

\subsection{Improvements of the Bennett--Hoeffding exponential bound}\label{BH}
It is seen from Figure~\ref{fig:subG} that the new bounds $P_2$ and $P_3$, and even the optimal bound $P_\opt$, provide only relatively limited improvements over the baseline sub-Gaussian bound $P_1$. 

In this subsection, it will be shown that the corresponding improvements over the well-known and widely used Bennett--Hoeffding exponential bound can be arbitrarily large (in a relative sense) in certain settings. 

Here it is still assumed that \eqref{eq:Y} holds, with independent zero-mean r.v.'s $Y_1,\dots,Y_n$. However, instead of the sub-Gaussian condition \eqref{eq:subG}, we now assume that 
\begin{equation*}
	Y_i\le b
\end{equation*}
for some real $b>0$ and all $i\in[n]$. We will also use notation \eqref{si^2}, but now with 
\begin{equation*}
	\si_i^2:=\Var Y_i=\E Y_i^2,
\end{equation*}
rather with $\si_i^2$ being a sub-Gaussian proxy variance of $Y_i$. 

It follows e.g.\ from \cite[Theorem~2]{pin-utev-exp} that, under the above conditions on \break  $Y,Y_1,\dots,Y_n$, the best possible upper bound on $\E e^{\la Y}$ is given by the inequality 
\begin{equation*}
	\E e^{\la Y}\le\mu_{\si,b}(\la):=\exp\Big\{\frac{\si^2}{b^2}(e^{\la b}-1-\la b)\Big\},
\end{equation*}
for each real $\la\ge0$. Thus, we have the Markov bound on $\P(Y\ge y)$: 
\begin{equation*}
	\P(Y\ge y)\le\frac{\mu_{\si,b}(\la)}{v_y(\la)},
\end{equation*}
where $v_y(\la)=e^{\la y}$, as in \eqref{eq:mu_la,v_la}. 
%The latter attains its minimum 
Minimizing the latter bound on $\P(Y\ge y)$ in $\la\ge0$, we get 
\begin{equation}\label{eq:BH}
	\P(Y\ge y)\le P_\BH(y,\si,b):=\frac{\mu_{\si,b}(\la_{y,\si,b})}{v_y(\la_{y,\si,b})},
\end{equation}
where 
\begin{equation*}
	\la_{y,\si,b}:=\frac1b\,\ln\Big(1+\frac{by}{\si^2}\Big), 
\end{equation*}
so that
\begin{equation}\label{eq:mu_BH}
	\mu_{\si,b}(\la_{y,\si,b})
	=\exp\Big\{\frac yb\Big[1-\frac{\si^2}{by}\,\ln\Big(1+\frac{by}{\si^2}\Big)\Big]\Big\}
\end{equation}
and 
\begin{equation}\label{eq:v_BH}
	v_y(\la_{y,\si,b})=\exp\Big\{\frac yb\,\ln\Big(1+\frac{by}{\si^2}\Big)\Big\}. 
\end{equation}
The bound $P_\BH(y,\si,b)$ on $\P(Y\ge y)$ in \eqref{eq:BH} is the famous and widely used Bennett~\cite{bennett}--Hoeffding~\cite{hoeff63} bound. 

Since the Bennett--Hoeffding (BH) bound is a species of the Markov bound, it can be improved using Theorem~\ref{th:}, just as the sub-Gaussian bound was improved in Propositions~\ref{prop:subG}, \ref{prop:subG1}, and \ref{prop:subG2} of Subsection~\ref{subG}. 
Here we will only consider the simplest of such improvements of the BH bound, based on \eqref{eq:right,q} (cf.\ \eqref{eq:P_3}), even though this improvement is not the best possible: 
\begin{proposition}\label{prop:BH}
For all real $y>0$, 
\begin{equation}\label{eq:BHimprov}
	\P(Y\ge y)\le %B_\opt(t):=
	P_{\BH,1}(y,\si,b):=\min\Big(1,\frac{\mu_{\si,b}(\la_{y,\si,b})-1}{v_y(\la_{y,\si,b})-1-\ln v_y(\la_{y,\si,b})}\Big),
%	P_3(t):=q_{\mu_t,v_t}=\min\Big(1,\frac{\mu_t-1}{v_t-1-\ln v_t}\Big) 
%	=\min\Big(1,\frac{e^{t^2/2}-1}{e^{t^2}-1-t^2}\Big), 
\end{equation}
where $\mu_{\si,b}(\la_{y,\si,b})$ and $v_y(\la_{y,\si,b})$ are as in \eqref{eq:mu_BH} and \eqref{eq:v_BH}. 
\end{proposition} 

Suppose now that 
\begin{equation*}
	\frac yb=c\quad\text{and}\quad \frac{\si^2}{by}\le e^{-C/c},
\end{equation*}
where $c$ and $C$ are positive real numbers. Then $\mu_{\si,b}(\la_{y,\si,b})
\le\exp\big\{\frac yb\big\}=e^c$ and $v_y(\la_{y,\si,b})\ge e^C$, so that 
\begin{equation*}
	\frac{P_{\BH,1}(y,\si,b)}{P_\BH(y,\si,b)}\le\frac{e^c-1}{e^c}\frac{e^C}{e^C-1-C}, 
\end{equation*}
because $\frac{\mu-1}\mu$ is increasing in $\mu>0$ and $\frac v{v-1-\ln v}$ is decreasing in $v>1$. 
So,
the ratio $P_{\BH,1}(y,\si,b)/P_\BH(y,\si,b)$ of the improved BH bound $P_{\BH,1}(y,\si,b)$ to the original BH bound $P_\BH(y,\si,b)$ can be however small if $c$ is small enough and $C$ is bounded away from $0$. Conditions with $C$ not small and $c$ not large arise in settings when possibly heavy tails of the distributions of the $Y_i's$ must be appropriately truncated -- see e.g.\ \cite{heyde67,pin81}. 

%\begin{equation}\label{eq:BHimprov}
%	P_{\BH,1}(y,\si,b)\le\frac{c-1}{e^C-1-C}\Big),
%\end{equation}
%
%It follows from \eqref{eq:BHimprov} that 
%
%\cite{heyde67,pin81}

\bigskip
\hrule\medskip

The results of Subsections~\ref{subG} and \ref{BH} can be extended to the case when the $Y_i's$ are (super)martingale-differences; cf.\ e.g.\ \cite[Section~8]{pin94}. 

\section{Proofs}\label{proofs}

\begin{proof}[Proof of Theorem~\ref{th:}] This proof is implicitly based on a duality argument; see e.g.\ \cite{kemper-dual,pin98}. 

Note that the probabilities $\P\big(\frac X{G_X}\ge v\big)$ and $\P\big(\frac X{G_X}\le v\big)$ will not change if we replace there $X$ by $X/u$, for any positive real $u$. So, without loss of generality we may and shall assume that $G_X=1$, that is, 
\begin{equation}\label{eq:G_X=1}
	\E\ln X=0, 
\end{equation}
so that the probabilities $\P\big(\frac X{G_X}\ge v\big)$ and $\P\big(\frac X{G_X}\le v\big)$ become simply $\P(X\ge v)$ and $\P(X\le v)$. 
 
Take now any positive real $v$ and any positive real $z\ne v$, and for all real $x>0$ let  
\begin{equation*}
	g(x):=ax-b\ln x+c,
\end{equation*}
where 
\begin{equation*}
	a:=a(z):=\frac{1/v}{h(r)},\quad b:=b(z):=az,\quad c:=c(z):=az\ln\frac ze, 
\end{equation*}
\begin{equation*}
	h(r):=1-r+r\ln r,\quad r:=z/v. %\in(0,1). 
\end{equation*}
Note that the function $h$ is convex on $(0,\infty)$, with $h(1)=0=h'(1)$. So, $h>0$ on $(0,\infty)\setminus\{1\}$ and hence $a>0$ and $b>0$. Therefore, the function $g$ is convex on $(0,\infty)$. Moreover, 
\begin{equation*}
	g(z)=g'(z)=0 \quad\text{and}\quad g(v)=1. 
\end{equation*}

So, if $0<z<v$, then $g(x)\ge\ii{x\ge v}$ for all real $x>0$, where $\ii\cdot$ denotes the indicator. Hence, in view of \eqref{eq:G_X=1},  
\begin{equation}\label{eq:P(X>v)<R}
\begin{aligned}
	\P(X\ge v)&\le\E g(X)=a\E X-b\E\ln X+c=a\,\mu+c \\ 
	&=R_z(v):=\frac{\mu-z+z\ln z}{v-z+z\ln z-z\ln v}\quad\text{if}\quad 0<z<v. 
\end{aligned}
\end{equation}
Similarly, if $0<v<z$, then $g(x)\ge\ii{x\le v}$ for all real $x>0$, whence 
\begin{equation}\label{eq:P(X<v)<R}
	\P(X\le v)\le R_z(v)\quad\text{if}\quad 0<v<z. 
\end{equation}
Recalling the conditions $\mu>1$ in \eqref{eq:EX,ElnX} and $q_1=1$ in the statement of part~\eqref{IV} of Theorem~\ref{th:}, as well as the fact that no probability can exceed $1$, and then substituting $1$ for $z$ in \eqref{eq:P(X>v)<R} and \eqref{eq:P(X<v)<R}, we get part~\eqref{IV} of Theorem~\ref{th:}. 

To prove part~\eqref{I} of Theorem~\ref{th:}, consider separately the two cases: $v\in(\mu,\infty)$ and $v\in(0,1)$. 

If $v\in(\mu,\infty)$, then the function $F\colon(0,\infty)\to\R$ is concave, with $F(0+)=-\infty<0$ and $F(1)=(\mu-1)\ln v>0$ (since $v>\mu>1$). So, indeed there is exactly one root $z=z_v\in(0,1)$ of equation~\eqref{eq:F}. 
Next, from the equality $F(z_v)=0$ we get $\ln z_v=(\mu - z_v)\ln(v)/(\mu - v)$. Substituting this expression for $\ln z_v$ into the expression for $R_z(v)$ in \eqref{eq:P(X>v)<R} and recalling the definition of $p_v$ in \eqref{eq:right}, we get 
\begin{equation}\label{eq:R=p}
	R_{z_v}(v)=p_v. 
\end{equation}
Therefore and because here 
\begin{equation}\label{eq:bet,right}
0<z_v<1<\mu<v, 	
\end{equation}
we see that the inequality in \eqref{eq:right} follows by \eqref{eq:P(X>v)<R}. 

The case $v\in(0,1)$ is similar (to the case $v\in(\mu,\infty)$). Indeed, if $v\in(0,1)$, then the function $F\colon(0,\infty)\to\R$ is convex, with $F(\mu)=(v-\mu)\ln\mu<0$ (since $\mu>1$) and $F(\infty-)=\infty>0$ (since $v>\mu>1$). So, indeed there is exactly one root $z=z_v\in(\mu,\infty)$ of equation~\eqref{eq:F}. Of course, equality \eqref{eq:R=p} holds for $v\in(0,1)$ as well. Therefore and because here 
\begin{equation}\label{eq:bet,left}
0<v<1<\mu<z_v, 	
\end{equation}
we see that the inequality in \eqref{eq:left} follows by \eqref{eq:P(X<v)<R}.

Also, in view of \eqref{eq:bet,right} and \eqref{eq:bet,left}, in either one of the cases $v\in(\mu,\infty)$ and $v\in(0,1)$, $\mu$ is strictly between $v$ and $z_v$, whence $p_v=\dfrac{\mu-z_v}{v-z_v}\in(0,1)$. 

Thus, part~\eqref{I} of Theorem~\ref{th:} is proved. 

To prove part~\eqref{II} of Theorem~\ref{th:}, note first that, in view of the just proved inclusion $p_v\in(0,1)$, there does exist a r.v.\ $X$ as in \eqref{eq:X_v}. For such a r.v.\ $X$, we have $A_X=\E X=\mu$ and $\ln G_X=\E\ln X=\dfrac{F(z_v)}{v-z_v}=0$, by the definition of $z_v$, so that the condition $A_X/G_X=\mu$ holds. Also, again in view of \eqref{eq:bet,right} and \eqref{eq:bet,left}, we have $z_v<v$ if $v\in(\mu,\infty)$, and $z_v>v$ if $v\in(0,1)$. So, for 
any r.v.\ $X$ as in \eqref{eq:X_v}, the inequalities in \eqref{eq:right} and \eqref{eq:left} turn into the equalities; that is, the upper bound $p_v$ in the inequalities in \eqref{eq:right} and \eqref{eq:left} is exact, as it is attained for $X$ as in \eqref{eq:X_v}. This proves part~\eqref{II} of Theorem~\ref{th:}. 

Next, consider part~\eqref{III} of Theorem~\ref{th:}. Note that the function $\rho_\mu$ is nonincreasing on $\R$ and $\rho_\mu\le1$ on $\R$. Also, by part~\eqref{II} of Theorem~\ref{th:} and the definition of $p_v$ in \eqref{eq:right}, for $v\in(\mu,\infty)$ we have $\rho_\mu(v)=p_v\to1$ as $v\downarrow\mu$, because $\mu>1$ and $z_v<1$. So, $\rho_\mu(\mu+)=1$ and hence $1\ge\rho_\mu(v)\ge\rho_\mu(\mu+)=1$ for all $v\in(-\infty,\mu]$. This proves \eqref{eq:rho=1}. 

Further, the function $\la_\mu$ is nondecreasing on $\R$ and $\la_\mu\le1$ on $\R$. Also, by part~\eqref{II} of Theorem~\ref{th:}, % and the definition of $p_v$ in \eqref{eq:right}
for $v\in(0,1)$ we have $\la_\mu(v)=p_v$. Let now $v\uparrow1$. Then $|(v-\mu)\ln z_v|$ is bounded away from $0$, because $\mu>1$ and $z_v>\mu$. So, in view of \eqref{eq:F} and the condition $F(z_v)=0$, $|(\mu-z_v)\ln v|$ is bounded away from $0$. So, for $v\uparrow1$ we have $z_v\to\infty$ and hence $\la_\mu(v)=p_v\to1$, again by the definition of $p_v$ in \eqref{eq:right}. Therefore, $\la_\mu(1-)=1$ and hence $1\ge\la_\mu(v)\ge\la_\mu(1-)=1$ for all $v\in[1,\infty)$. This proves \eqref{eq:la=1}. 

Concerning the last, non-attainment clause in part~\eqref{III} of Theorem~\ref{th:}: 
If $\P\big(\frac X{G_X}\ge v\big)=1$ for some $v\in[1,\mu]$, then $\P\big(\frac X{G_X}\ge1\big)=1$, which implies that $\P(X=G_X)=1$, which contradicts the inequality in \eqref{eq:EX,ElnX}. Similarly, if $\P\big(\frac X{G_X}\le v\big)$ for some $v\in[1,\mu]$, then $\P\big(\frac X{G_X}\le\mu\big)=1$, which implies that $\frac{A_X}{G_X}=\E\frac X{G_X}\le\mu$, with the strict inequality $\frac{A_X}{G_X}<\mu$ (contradicting the definition of $\mu$ in \eqref{eq:EX,ElnX}) unless $\P\big(\frac X{G_X}=\mu\big)=1$. But the latter equality implies $\P(X=c)=1$ for some real $c>0$, which contradicts the inequality in \eqref{eq:EX,ElnX} (since the function $\ln$ is strictly concave).  

Thus, for each $v\in[1,\mu]$, the exact upper bound, $1$, on either one of the two tail probabilities, $\P\big(\frac X{G_X}\ge v\big)$ and $\P\big(\frac X{G_X}\le v\big)$, is not attained.

Finally, concerning part~\eqref{V} of Theorem~\ref{th:}: 
%The condition $\mu:=A_X/G_X$ in %Theorem~\ref{th:}
%\eqref{eq:EX,ElnX} can be replaced by the $A_X/G_X\le\mu$. 
%This 
%follows from the proof of Theorem~\ref{th:} (in Section~\ref{proofs}). More specifically, 
Given only the condition $A_X/G_X\le\mu$ (which means that $\E X\le\mu$ when \eqref{eq:G_X=1} is assumed), the second equality sign in \eqref{eq:P(X>v)<R} can be replaced by $\le$, since $a>0$. So, the inequality $\P(X\ge v)\le R_z(v)$ will continue to hold when $0<z<v$. Similarly, \eqref{eq:P(X<v)<R} will continue to hold. %, under the assumption $A_X/G_X\le\mu$. 

Theorem~\ref{th:} is now completely proved.
\end{proof}

\begin{proof}[Proof of Proposition~\ref{prop:p,q}]
For brevity, let 
\begin{equation*}
	p:=p_v,\quad q:=q_v,\quad z:=z_v,
\end{equation*}
and then 
\begin{equation*}
	\de_\mu:=\mu-1\downarrow0,\quad\de_v:=v-1\to0,\quad\de_z:=z-1, 
\end{equation*}
so that, by \eqref{eq:F}, \eqref{eq:right}, \eqref{eq:right,q}, and the condition $q_v<1$, 
\begin{equation}\label{eq:F,1}
	(\de_v-\de_\mu)\ln(1+\de_z)+(\de_\mu-\de_z)\ln(1+\de_v)=0, 
\end{equation}
\begin{equation}\label{eq:p=}
	p=\frac{\de_\mu-\de_z}{\de_v-\de_z}
\end{equation}
and 
\begin{equation*}%\label{eq:right,q}
	q=\frac{\de_\mu}{\de_v-\ln(1+\de_v)}\sim\frac{\de_\mu}{\de_v^2/2},  
\end{equation*}
whence  
\begin{equation}\label{eq:de_mu sim}
	\de_\mu\sim q\de_v^2/2=o(\de_v). 
\end{equation}
Therefore, $\de_v-\de_\mu\sim\de_v\sim\ln(1+\de_v)$ and hence \eqref{eq:F,1} implies $\ln(1+\de_z)\sim\de_z-\de_\mu$. Since $\de_\mu\to0$, it follows that $\de_z\to0$. (Otherwise, without loss of generality we have
$\ln(1+\de_z)\sim\de_z$, which implies $\de_z\to0$, since $\ln(1+u)<u$ for all $u\in(-1,\infty)\setminus\{0\}$ and $\ln(1+u)\not\sim u$ as $u\downarrow-1$ or $u\to\infty$.) 
% without loss of generality we have $\de_z\to c$ for some $c\in[-1,\infty]\setminus\{0\}$

Now \eqref{eq:F,1} and \eqref{eq:de_mu sim} yield   
\begin{equation}\label{eq:F,2}
	\Big(\de_v-\frac{q\de_v^2}{2+o(1)}\Big)\Big(\de_z-\frac{\de_z^2}{2+o(1)}\Big)
	+\Big(\frac{q\de_v^2}{2+o(1)}-\de_z\Big)\Big(\de_v-\frac{\de_v^2}{2+o(1)}\Big)=0, 
\end{equation}
which simplies to
\begin{equation}\label{eq:F,3}
	\Big(1+\frac{(1-q)\de_v}{2+o(1)}\Big)\Big(\de_z-\frac{\de_z^2}{2+o(1)}\Big)
	+\Big(\frac{q\de_v^2}{2+o(1)}-\de_z\Big)=0  
\end{equation}
and then to 
\begin{equation}\label{eq:qe}
	\de_z^2-(1-q)(1+o(1))\de_v\de_z-q(1+o(1))\de_z^2=0. 
\end{equation}
Also, by part~\eqref{I} of Theorem~\ref{th:}, $\de_z\de_v<0$. 
So, ``solving'' the ``quadratic'' equation \eqref{eq:qe}, we get 
\begin{align*}
	\frac{\de_z}{\de_v}&=\frac{(1-q)(1+o(1))-\sqrt{(1-q)^2(1+o(1))+4q(1+o(1))}}2 \\ 
	&=\frac{(1-q)(1+o(1))-(1+q)(1+o(1))}2=-q+o(1).  
\end{align*}
%
%\begin{align*}
%	\frac{\de_z}{\de_v}&=\frac{(1-q)(1+o(1))}2-\sqrt{\frac{(1-q)^2(1+o(1))}4+q(1+o(1))} \\ 
%	&=\frac{(1-q)(1+o(1))}2-\frac{(1+q)(1+o(1))}2=-q+o(1).  
%\end{align*}
%
Now \eqref{eq:p,q} follows by \eqref{eq:p=} and \eqref{eq:de_mu sim}.  
\end{proof}

The proof of Proposition~\ref{prop:L} is based in part on the following lemmas. 
\begin{lemma}\label{lem:1} 
We have $\tz_v\in(0,\infty)$ for all $v\in(\mu,\infty)\cup(0,1)$. Also, 
\begin{equation}\label{eq:lem1}
\tz_v
	\begin{cases}
	<v&\text{ if }v\in(\mu,\infty), \\ 
	>v&\text{ if }v\in(0,1). 
	\end{cases}
\end{equation}
\end{lemma}

\begin{lemma}\label{lem:2} %\reverse-reverse-jensen\Mathematica\Untitled-3.nb
For $v\in(\mu,\infty)\cup(0,1)$, 
\begin{equation}\label{eq:H>0}
	%H(v):=
	H_\mu(v):=\frac{G_\mu(v)}{v-1}>0,
\end{equation}
where 
\begin{equation*}
	%G(v):=
	G_\mu(v):=F(\tz_v)
	=\mu\ln v+(v-\mu)\ln\frac{v-\mu}{\ln v}+\mu-v, 
\end{equation*}
with $\tz_v$ as in \eqref{eq:tz_v}. 
\end{lemma}

\begin{lemma}\label{lem:3} 
We have $\tz_v\in(0,\infty)$ for all $v\in(\mu,\infty)\cup(0,1)$. Also, 
\begin{equation}\label{eq:lem3}
z_v
	\begin{cases}
	<\tz_v&\text{ if }v\in(\mu,\infty), \\ 
	>\tz_v&\text{ if }v\in(0,1). 
	\end{cases}
\end{equation}
\end{lemma}

\begin{proof}[Proof of Lemma~\ref{lem:1}] 
That $\tz_v\in(0,\infty)$ for all $v\in(\mu,\infty)\cup(0,1)$ follows immediately from the definition of $\tz_v$ in \eqref{eq:tz_v} and the condition $\mu>1$. 
Next, for each $v\in(\mu,\infty)\cup(0,1)$, each of the two inequalities in \eqref{eq:lem1} can be rewritten as $l(v)>0$, where 
\begin{equation*}
	l(v):=v\ln v-v+\mu. 
\end{equation*}
The function $l$ is convex on $(0,\infty)$, with $l(1)=\mu-1>0$ and $l'(1)=0$. So, $l(v)>0$ for $v\in(0,1)\cup(1,\infty)$ and hence for $v\in(\mu,\infty)\cup(0,1)$, which completes the proof of Lemma~\ref{lem:1}.  
%
%If $v\in(\mu,\infty)$, then $v>\mu>1$ and hence the condition $\tz_v<v$ cab be rewritten as $l(v)>0$, where 
%\begin{equation}
%	l(v):=v\ln v-v+\mu. 
%\end{equation}
%The function $l$ is convex on $(0,\infty)$, with $l(\mu)=\mu\ln\mu>0$ and $l'(\mu)=\ln\mu>0$. So, indeed $l(v)>0$ and hence $\tz_v<v$ if $v\in(\mu,\infty)$. 
%
%If $v\in(0,1)$, then the condition $\tz_v>v$ cab be rewritten as $l(v)<0$. Also, $l'(v)=\ln v<0$ for $v\in(0,1)$, so that $l$ is decreasing on $(0,1]$, with $l(1)=\mu-1>0$. 
%The function $l$ is convex on $(0,\infty)$, with $l(\mu)=\mu\ln\mu>0$ and $l'(\mu)=\ln\mu>0$. So, indeed $l(v)>0$ and hence $\tz_v<v$ if $v\in(\mu,\infty)$. 
% 
\end{proof}

\begin{proof}[Proof of Lemma~\ref{lem:2}] 
Note first that the partial derivative of $G_\mu(v)$ in $\mu$ is $\ln v-\ln\tz_v$. So, by Lemma~\ref{lem:1}, $G_\mu(v)$ is increasing in $\mu\in[1,v)$ if $v\in(1,\infty)$ and decreasing in $\mu\in[1,\infty)$ if $v\in(0,1)$. 
%So, if $v\in(0,1)$, then $G_\mu(v)<G_{v_+}(v)=v\ln v<0$ and therefore the inequality in \eqref{eq:H>0} holds. 
It follows that $G_\mu(v)>G_1(v)$ if $v\in(\mu,\infty)$ and $G_\mu(v)<G_1(v)$ if $v\in(0,1)$. 
%Also, for all $v\in(\mu,\infty)\cup(0,1)$, $v-\mu$ is of the same sign as $v-1$. 

%Suppose now that $v\in(\mu,\infty)$. Then $G_\mu(v)>G_1(v)$, since $G_\mu(v)$ is increasing in $\mu\in[1,v)$. 
So, to complete the proof of Lemma~\ref{lem:2}, it is enough to show that 
\begin{equation}\label{eq:H_1>0}
	H(v):=H_1(v)=\frac{\ln v}{v-1}+\ln\frac{v-1}{\ln v}-1\overset{\text(?)}>0\quad\text{ if }v\in(0,1)\cup(1,\infty). 
\end{equation}
We have %\reverse-reverse-jensen\Mathematica\Untitled-3.nb 
\begin{equation*}
	H'(v):=\frac{(v-1-\ln v) (v\ln v-v+1)}{(v-1)^2\, v \ln v}, 
\end{equation*}
which is easily seen to be of the same sign as $v-1$ for all $v\in(0,1)\cup(1,\infty)$. So, $H(v)$ is decreasing in $v\in(0,1)$ and 
increasing in $v\in(1,\infty)$. Also, $H(1+)=H(1-)=0$. Thus, \eqref{eq:H_1>0} is true, which completes the proof of Lemma~\ref{lem:2}. 
\end{proof}

\begin{proof}[Proof of Lemma~\ref{lem:3}] 
Consider first the case $v\in(\mu,\infty)$. Then, as was noted in the proof of part~\eqref{I} of Theorem~\ref{th:}, the function $F\colon(0,\infty)\to\R$ is concave. Also, $F(v)=0$ and, by the definition of $z_v$ in part~\eqref{I} of Theorem~\ref{th:}, $F(z_v)=0$. Further, by Lemma~\ref{lem:2}, $F(\tz_v)>0$. Therefore and in view of the concavity of $F$, $\tz_v$ is strictly between $v$ and $z_v$. But, by Lemma~\ref{lem:1}, here $\tz_v<v$. So, the first inequality in \eqref{eq:lem3} is proved. 

The second case, with $v\in(0,1)$, is treated similarly. In this case, the function $F\colon(0,\infty)\to\R$ is convex and by Lemma~\ref{lem:2}, $F(\tz_v)<0$. Here we still have  $F(v)=0$ and $F(z_v)=0$, whence again $\tz_v$ is strictly between $v$ and $z_v$. But, by Lemma~\ref{lem:1}, here $\tz_v>v$. So, the second inequality in \eqref{eq:lem3} is proved as well. 
\end{proof}

\begin{proof}[Proof of Proposition~\ref{prop:L}] 
Take any $v\in(\mu,\infty)\cup(0,1)$. 
By the definition of $z_v$ in part~\eqref{I} of Theorem~\ref{th:}, $F(z_v)=0$, that is, $(v-\mu)\ln z_v+(\mu-z_v)\ln v=0$. Dividing the latter equality by $v-\mu$ and recalling the definition of $\tz_v$ in \eqref{eq:tz_v}, rewrite the defining condition on $z_v$ as 
\begin{equation}\label{eq:L}
	\ln z_v-\frac{z_v}{\tz_v}=-\frac\mu{\tz_v}. 
\end{equation}
Exponentiating both sides of \eqref{eq:L} and then dividing the resulting expressions by $-\tz_v$, rewrite \eqref{eq:L} as 
\begin{equation*}
t_v e^{t_v}=u_v,	
\end{equation*}
where 
\begin{equation*}
	t_v:=-\frac{z_v}{\tz_v}\quad\text{and}\quad u_v:=-\frac{e^{-\mu/\tz_v}}{\tz_v}. 
\end{equation*}
Note also that $te^t\in(-1/e,0)$ for $t\in(-\infty,-1)\cup(-1,0)$. 
So, in view of the description of the branches $W_0$ and $W_{-1}$ of Lambert's $W$ function given at the end of the statement of Proposition~\ref{prop:L}, it remains to check that $t_v\in(-1,0)$ if $v\in(\mu,\infty)$ and $t_v\in(-\infty,-1)$ if $v\in(0,1)$; but these conditions on $t_v$ follow immediately by Lemma~\ref{lem:3}. Proposition~\ref{prop:L} is proved. 
\end{proof}

%% The Appendices part is started with the command \appendix;
%% appendix sections are then done as normal sections
%% \appendix

%% \section{}
%% \label{}

%% References
%%
%% Following citation commands can be used in the body text:
%% Usage of \cite is as follows:
%%   \cite{key}          ==>>  [#]
%%   \cite[chap. 2]{key} ==>>  [#, chap. 2]
%%   \citet{key}         ==>>  Author [#]

%% References with bibTeX database:

%%%\bibliographystyle{model1-num-names}
%%%\bibliography{<your-bib-database>}

\bibliographystyle{abbrv}
%\bibliographystyle{ims}
%\bibliography{are.citations}
%\bibliography{citat}

%\bibliography{citations}

\bibliography{P:/pCloudSync/mtu_pCloud_02-02-17/bib_files/citations01-09-20}
%\bibliography{C:/Users/Iosif/Dropbox/mtu/bib_files/citations12.13.12}
%\bibliography{C:/Users/Iosif/Documents/mtu_home01-30-10/bib_files/citations}
%\bibliography{C:/Users/Iosif/Documents/mtu_home12-22-08/bib_files/citations}

\end{document}